\documentclass[11pt]{amsart}%
\usepackage{palatino, mathpazo}
\usepackage[mathscr]{eucal}
\usepackage{amssymb}
\usepackage{graphicx}
\usepackage{color}
\usepackage{amsfonts}
\usepackage{amsmath}
\usepackage{amssymb,latexsym}%
\providecommand{\U}[1]{\protect \rule{.1in}{.1in}}
\newtheorem{theorem}{Theorem}[section]

\newtheorem{lemma}[theorem]{Lemma}

\newtheorem{remark}[theorem]{Remark}

\numberwithin{equation}{section}
\begin{document}
\title[Mean field equations on tori]{Non-existence of solutions for a mean field equation on flat tori at critical
parameter $16\pi$}
\author{Zhijie Chen}
\address{Department of Mathematical Sciences, Yau Mathematical Sciences Center,
Tsinghua University, Beijing, 100084, China }
\email{zjchen@math.tsinghua.edu.cn}
\author{Ting-Jung Kuo}
\address{Department of Mathematics, National Taiwan Normal University,
Taipei 11677, Taiwan }
\email{tjkuo1215@gmail.com}
\author{Chang-Shou Lin}
\address{Taida Institute for Mathematical Sciences (TIMS), Center for Advanced Study in
Theoretical Sciences (CASTS), National Taiwan University, Taipei 10617, Taiwan }
\email{cslin@math.ntu.edu.tw}

\begin{abstract}
It is known from \cite{LW} that the solvability of the mean field equation $\Delta u+e^{u}=8n\pi \delta_{0}$ with $n\in\mathbb{N}_{\geq 1}$ on
a flat torus $E_{\tau}$ essentially depends on the geometry of $E_{\tau}$. A conjecture is the non-existence of solutions for this equation if $E_{\tau}$ is a rectangular torus, which was proved for $n=1$ in \cite{LW}. For any $n\in \mathbb{N}_{\geq2}$, this conjecture seems
challenging from the viewpoint of PDE theory.
In this paper, we prove this conjecture for $n=2$ (i.e. at critical
parameter $16\pi$).

\end{abstract}
\maketitle

\section{Introduction}

Let $E_{\tau}:=\mathbb{C}/\Lambda_{\tau}$ be a flat torus on the plane, where
$\Lambda_{\tau}=\mathbb{Z+Z}\tau$, $\tau \in \mathbb{H}=\left \{  \tau \text{
}|\text{ }\operatorname{Im}\tau>0\right \}  $. Consider the following mean
field equation with a parameter $\rho>0$:%
\begin{equation}
\Delta u+e^{u}=\rho \cdot \delta_{0}\  \  \text{on}\ E_{\tau},\label{eq0}%
\end{equation}
where $\delta_{0}$ is the Dirac measure at the origin $0$. Equation (\ref{eq0}) has a geometric origin (cf. \cite{CLW}). In conformal
geometry, for a solution $u(x)$, the new metric $ds^{2}=e^{u(x)}|dx|^{2}$ has
positive constant curvature. Since the RHS has singularities, $ds^{2}$ is a
metric with \textit{conic singularity}. Equation (\ref{eq0}) also appears in
statistical physics as the equation for the mean field limit of the Euler flow
in Onsager's vortex model (cf. \cite{CLMP}), hence its name. Recently equation
(\ref{eq0}) was shown to be related to the self-dual condensates of the
Chern-Simons-Higgs equation in superconductivity. We refer the readers to
\cite{CLW4,Choe,Eremenko1,LY,NT1} and references therein for recent developments of
related subjects of equation (\ref{eq0}).

When $\rho \not \in 8\pi \mathbb{N}$, it can be proved that solutions of
(\ref{eq0}) have \textit{uniform a priori} bounds in $C_{loc}^{2}(E_{\tau
}\backslash \{0\})$ and hence the topological Leray-Schauder degree $d_{\rho}$
is well-defined; see \cite{Bartolucci,CL-1}. Recently, Chen and the second
author \cite{CL3} proved that $d_{\rho}=m$ for any $m\in \mathbb{N}_{\geq1}$
and $\rho \in(8\pi(m-1),8\pi m)$. Consequently, \textit{equation (\ref{eq0}) always
has solutions when }$\rho \not \in 8\pi \mathbb{N}$\textit{, no matter with the
geometry of the torus} $E_{\tau}$.

However when $\rho \in8\pi \mathbb{N}_{\geq1}$, \textit{a priori }bounds for
solutions of (\ref{eq0}) might not exist, and the existence of solutions
becomes an intricate question. In this paper, we consider this mean field
equation at critical parameters $\rho=8n\pi$ (\cite{CLW,LW,CLW2}):
\begin{equation}
\Delta u+e^{u}=8n\pi \delta_{0}\  \  \text{on}\ E_{\tau},\label{eq1}%
\end{equation}
where $n\in \mathbb{N}_{\geq1}$. The case $n=1$ was first studied by Wang and
the second author \cite{LW}, where they discovered that \emph{the solvability of
equation (\ref{eq1}) essentially depends on the moduli $\tau$ of the
torus $E_{\tau}$}, a surprising phenomena which does not appear for
non-critical parameter $\rho$'s. For example, they proved that when $\tau \in
i\mathbb{R}^{+}$ (i.e. $E_{\tau}$ is a rectangular torus), equation (\ref{eq1})
with $n=1$ has {\it no} solution; while for $\tau=\frac{1}{2}+\frac{\sqrt{3}}{2}i$
(i.e. $E_{\tau}$ is a rhombus torus), equation (\ref{eq1}) with $n=1$ has solutions. Later, the case $n=1$ was thoroughly investigated in \cite{new1}.

To settle this challenging
problem for $n\geq 2$, Chai-Lin-Wang \cite{CLW} and subsequently Lin-Wang \cite{CLW2} studied it from the viewpoint of algebraic geometry. They
developed a theory to connect this PDE problem with hyper-elliptic curves and modular forms. Among other things, they proposed the following conjecture.\medskip

\noindent \textbf{Conjecture. }\cite{CLW2} \textit{When $\tau \in
i\mathbb{R}^{+}$, i.e. $E_{\tau}$ is a rectangular torus, equation (\ref{eq1}) has no solutions for any $n\geq 2$.} \medskip

This paper is the first in our project devoted to studying the existence (or
non-existence) problem of equation (\ref{eq1}) for $n\geq2$.
The purpose of this paper is to confirm the conjecture for $n=2$.

\begin{theorem}
\label{thm1}Suppose $\tau \in i\mathbb{R}^{+}$, i.e. $E_{\tau}$ is a rectangular
torus. Then equation (\ref{eq1}) with $n=2$ on $E_{\tau}$ has no solutions.
\end{theorem}

Theorem \ref{thm1} has important applications. In a forthcoming paper, we will apply Theorem \ref{thm1} (together with the modular form theory established in \cite{CLW2}) to prove the following existence result on rhombus tori.

\medskip
\noindent \textbf{Theorem A.}
{\it Let $\tau=\frac{1}{2}+ib$ with $b>0$. Then there exists $b^{\ast
}\in(\frac{\sqrt{3}}{2},\frac{6}{5})$ such that for any $b>b^{\ast}$, equation
(\ref{eq1}) with $n=2$ on $E_{\tau}$ has a solution.}
\medskip

Remark that Theorem A is almost optimal in the sense that if
$\tau=\frac{1}{2}+\frac{\sqrt{3}}{2}i$, then equation (\ref{eq1}) with $n=2$
on $E_{\tau}$ has \emph{no} solutions (as mentioned before, (\ref{eq1}) with $n=1$ on this $E_{\tau}$ has solutions. This shows why we need to discuss different $n$'s separately). See Theorem \ref{thm-000} in Section 3.

In PDE theory, a standard method of proving non-existence results is to
apply the Pohozaev identity; see \cite{new3} for example. Obviously, this
method by Pohozaev identity does not work here. Our proof is based on the fact that equation (\ref{eq1}) can be viewed as an integrable system \cite{CLW}.

The paper is organized as follows. In Section \ref{non-existence}, we give a
short review of equation (\ref{eq1}) from the aspect of integrable system.
This point of view can reduce our existence problem to a couple equations
involving with Weierstrass elliptic functions. In Section 3, we prove this
couple equations have no solutions if $\tau \in i\mathbb{R}^{+}$. Our proof is elementary in the sense that only the basic theory of Weierstrass elliptic functions covered by the standard textbook (cf. \cite{A}) are used. This gives
the proof of Theorem \ref{thm1}.

\section{Overview of (\ref{eq1}) as an integrable system}

\label{non-existence}

In this section, we provide some basic facts about equation (\ref{eq1}) from
the viewpoint of integrable system; see \cite{CLW} for a complete discussion.
Throughout the
paper, we use the notations: $\omega_{0}=0$, $\omega_{1}=1$, $\omega_{2}=\tau$, $\omega_{3}=1+\tau$.

The Liouville theorem says that for any solution $u(z)$ to (\ref{eq1}), there
is a meromorphic function $f(z)$ defined in $\mathbb{C}$ such that%
\begin{equation}
u(z)=\log \frac{8|f^{\prime}(z)|^{2}}{(1+|f(z)|^{2})^{2}}. \label{502}%
\end{equation}
This $f(z)$ is called a developing map. Although $u$ is a doubly periodic
function, $f(z)$ is not an elliptic function. By differentiating (\ref{502}),
we have
\begin{equation}
u_{zz}-\frac{1}{2}u_{z}^{2}= \{ f;z \}:=\left(  \frac{f^{\prime \prime}%
}{f^{\prime}}\right)  ^{\prime}-\frac{1}{2}\left(  \frac{f^{\prime \prime}%
}{f^{\prime}}\right)  ^{2}. \label{new22}%
\end{equation}
Conventionally, the RHS of this identity is called the Schwarzian derivative
of $f(z)$, denoted by $\{ f;z \}$. By the classical Schwarzian theory, any two
developing maps $f_{1}$ and $f_{2}$ of the same solution $u$ must satisfy
\begin{equation}
f_{2}(z)=\gamma \cdot f_{1}(z):=\frac{af_{1}(z)+b}{cf_{1}(z)+d} \label{cc}%
\end{equation}
for some $\gamma=%
\begin{pmatrix}
a & b\\
c & d
\end{pmatrix}
\in SL(2,\mathbb{C})$. Furthermore, by substituting (\ref{cc}) into
(\ref{502}), a direct computation shows $\gamma \in SU(2)$, i.e.%
\begin{equation}
d=\bar{a},\text{ \ }c=-\bar{b}\text{ \ and \ }|a|^{2}+|b|^{2}=1. \label{new25}%
\end{equation}

As we mentioned above, $f(z)$ is not doubly periodic. But $f(z+w_{1})$ and
$f(z+w_{2})$ are also developing maps of the same $u(z)$ and then (\ref{cc})
implies the existence of $\gamma_{i}\in SU(2)$ such that
\begin{equation}
f(z+\omega_{1})=\gamma_{1}\cdot f(z)\  \  \text{and}\  \ f(z+\omega_{2}%
)=\gamma_{2}\cdot f(z). \label{new26}%
\end{equation}
After normalizing $f(z)$ by the action of some $\gamma \in SU(2)$,
(\ref{new26}) can be simplified by
\begin{equation}
f(z+\omega_{j})=e^{2\pi i\theta_{j}}f(z),\text{ \ }j=1,2, \label{cc-4}%
\end{equation}
for some $\theta_{j}\in \mathbb{R}$. We call a developing map $f$ satisfying
(\ref{cc-4}) a \emph{normalized developing map}.

A simple observation is that once $f$ satisfies (\ref{cc-4}), then for any
$\beta \in \mathbb{R}$, $e^{\beta}f(z)$ also satisfies (\ref{cc-4}). Therefore,
once we have a solution $u(z)$, then we get a $1$-parameter family of
solutions:%
\[
u_{\beta}(z):=\log \frac{8e^{2\beta}|f^{\prime}(z)|^{2}}{(1+e^{2\beta
}|f(z)|^{2})^{2}}. \label{24-1}%
\]
Clearly $u_{\beta}(z)$ blow up as $\beta \rightarrow \pm \infty$. More precisely,
$u_{\beta}(z)$ blow up at and only at any zeros of $f(z)$ as $\beta
\rightarrow+\infty$, and $u_{\beta}(z)$ blow up at and only at any poles of
$f(z)$ as $\beta \rightarrow-\infty$. For (\ref{eq1}), the blowup set of a
sequence of solutions $u_{\beta}$ consists of $n$ distinct points in $E_{\tau
}$. Hence $f(z)$ has zeros at $z=a_{i} \in E_{\tau}$, $i=1, \cdots,n$, and
poles at $z=b_{i} \in E_{\tau}$, $i=1, \cdots,n$. Furthermore, $\{ a_{1},
\cdots,a_{n} \}= \{ -b_{1}, \cdots,-b_{n} \}$ in $E_{\tau}$; see \cite{CLW}.
Since $\{ a_{i} \}$ and $\{ b_{i} \}$ are the zeros and poles of a meromorphic
function, we have
\begin{equation}
\label{eq29}a_{i} \neq a_{j}\  \text{for any}\ i\neq j\ ;\quad a_{i} \neq
-a_{j}\  \text{for any}\ i , j.
\end{equation}

In the sequel, we always assume $n=2$ in (\ref{eq1}). So $u_{\beta}$ has
exactly two blowup points as $\beta \rightarrow+\infty$, say $a$ and $b$. Then
(\ref{eq29}) and the well known Pohozaev identity imply that $a$ and $b$
satisfy
\begin{equation}
2G_{z}(a)=G_{z}(a-b),\quad2G_{z}(b)=G_{z}(b-a), \quad a \notin \{-a,\pm b\}.
\label{4-23-1}%
\end{equation}
where $G(z)=G(z|\tau)$ is the Green function of $-\Delta$ on the torus
$E_{\tau}$. See \cite{CL-1, CL3} for the Pohozaev identity. Since the Green
function $G(z)$ is even, $G_{z}(z)$ is odd and (\ref{4-23-1}) is equivalent
to
\begin{equation}
G_{z}(a)+G_{z}(b)=0,\; G_{z}(a)-G_{z}(b)-G_{z}(a-b)=0,\; a \notin \{-a,\pm b\}.
\label{1-17-3}%
\end{equation}
On the other hand, the Green function $G(z)$ can be written in terms of
Weierstrass elliptic functions, see \cite{LW}. In particular, we have
\begin{align}
-4\pi G_{z}(z)  &  =\zeta(z|\tau)-\eta_{1}(\tau)z+\frac{2\pi
i\operatorname{Im}z}{\operatorname{Im}\tau}\label{40}\\
&  =\zeta(z|\tau)-r\eta_{1}(\tau)-s\eta_{2}(\tau),\nonumber
\end{align}
where $z=r+s\tau$ with $r,s\in \mathbb{R}$. Here we recall that $\wp(z)=\wp(z|\tau)$ is the Weierstrass elliptic function with periods $\omega_1=1$ and $\omega_2=\tau$, defined by
$$\wp(z|\tau):=\frac{1}{z^2}+\sum_{\omega\in\Lambda_\tau\setminus\{0\}}
\left(\frac{1}{(z-\omega)^2}
-\frac{1}{\omega^2}\right),$$
and $\zeta(z)=\zeta(z|\tau):=-\int^z\wp(\xi|\tau)d\xi$ is the Weierstrass zeta function, which is an odd meromorphic function with two quasi-periods $\eta_{j}(z)$ (cf. \cite{Lang}):
\begin{equation}
\zeta(z+1|\tau)=\zeta(z|\tau)+\eta_{1}(\tau),\  \  \zeta(z+\tau|\tau
)=\zeta(z|\tau)+\eta_{2}(\tau). \label{quasi}%
\end{equation}

In view of \eqref{40}, the second
equation in \eqref{1-17-3} can be changed to
\begin{equation}
\zeta(a)-\zeta(b)-\zeta(a-b)=0. \label{2-17-3}%
\end{equation}
Next, we should apply the classical addition formula (cf. \cite{Lang}):
\[
\zeta(u+v)-\zeta(u)-\zeta(v)=\frac{1}{2}\frac{\wp^{\prime}(u)-\wp^{\prime}%
(v)}{\wp(u)-\wp(v)}%
\]
by taking $(u,v)=(a,-b)$. Then (\ref{2-17-3}) becomes
\[
\wp^{\prime}(a)+\wp^{\prime}(b)=0.
\]
Therefore, the Pohozaev identity (\ref{1-17-3}) is equivalent to
\begin{equation}
G_{z}(a)+G_{z}(b)=0\  \  \text{and}\  \  \wp^{\prime}(a)+\wp^{\prime}(b)=0,\quad a
\notin \{-a,\pm b\}. \label{new213}%
\end{equation}
Thus, we summarize the main result in this short overview as follows:
\textit{Suppose equation (\ref{eq1}) with $n=2$ has a solution $u$, then there
exist $a,b\in E_{\tau}$ such that (\ref{new213}) holds true.}

\section{Non-existence for $\tau \in i\mathbb{R}^{+}$}

\label{non-ex}

In this section, we want to prove the non-existence of solutions to%
\begin{equation}
\Delta u+e^{u}=16\pi \delta_{0}\  \text{ on}\;E_{\tau}, \label{0011-1-22-1}%
\end{equation}
if $\tau \in i\mathbb{R}^{+}$, i.e. $E_{\tau}$ is a rectangular torus. To prove
this non-existence result, it suffices to show that there are no pair $(a,b)$
in $E_{\tau}$ such that (\ref{new213}) holds. The proof for $\tau \in
i\mathbb{R}^{+}$ is really non-trivial, however, it is much simpler if $\tau=
e^{\frac{\pi i}{3}}$.

\begin{theorem}
\label{thm-000}Let $\rho:= e^{\pi i/3}=\frac{1}{2}+\frac{\sqrt{3}}{2}i$. Then
equation%
\begin{equation}
\Delta u+e^{u}=16\pi \delta_{0}\text{ \ on \ }E_{\rho} \label{mfe-002}%
\end{equation}
has no solutions.
\end{theorem}

\begin{proof}
Assume by contradiction that (\ref{mfe-002}) has a solution. Then there exist
$a,b \in E_{\rho}$ such that (\ref{new213}) holds, i.e.%
\[
G_{z}(a|\rho)+G_{z}(b|\rho)=0,\text{ }\wp^{\prime}(a|\rho)+\wp^{\prime}%
(b|\rho)=0\; \text{ and }\;a\notin \{-a,\pm b\}.
\]
It is known (cf. \cite{LW}) that $g_{2}(\rho)=0$ (see (\ref{new34}) for $g_2$) and $\wp(z|\rho)=\rho^{2}%
\wp(\rho z|\rho)$. Then by $\wp^{\prime}(a|\rho)^{2}=\wp^{\prime}(b|\rho)^{2}$
and (\ref{new34}) below, we obtain $\wp(a|\rho)^{3}=\wp(b|\rho)^{3}$, which
implies%
\[
\text{either \ }b=\pm \rho a\text{ \ or }b=\pm \rho^{2}a.
\]
On the other hand, $G(\rho z|\rho)=G(z|\rho)$ gives $\rho G_{z}(\rho
z|\rho)=G_{z}(z|\rho)$. Hence,%
\[
0=G_{z}(a|\rho)+G_{z}(b|\rho)=\left(  1\pm \rho^{-j}\right)  G_{z}%
(a|\rho)\text{ for some }j\in \{1,2\} \text{,}%
\]
which implies that $a$ is a critical point of $G(z|\rho)$ and so does $b$.
Recall from \cite{LW} that $G(z|\rho)$ has exactly five critical points $\{
\frac{1}{2}\omega_{1},\frac{1}{2}\omega_{2},\frac{1}{2}\omega_{3},\pm \frac
{1}{3}\omega_{3}\}$. So $a,b \in \{ \frac{1}{2}\omega_{1},\frac{1}{2}\omega
_{2},\frac{1}{2}\omega_{3},\pm \frac{1}{3}\omega_{3}\}$, a contradiction with
$\wp^{\prime}(a|\rho)+\wp^{\prime}(b|\rho)=0$ and $a\notin \{-a,\pm b\}$.
Therefore, (\ref{mfe-002}) has no solutions.
\end{proof}

From now on, we assume that $\tau\in i\mathbb{R}^+$, i.e. $E_{\tau}$ \textit{is a rectangle centered at the
origin}. Under this assumption, we will prove Theorem \ref{thm1}.

To prove Theorem \ref{thm1}, we will show that if $(a,b)$ is a solution of
\eqref{new213}, then both $a$ and $b$ lie in the same half plane, and then we
exclude this possibility by using the elementary properties of the Green
function $G$.

Our proof is elementary in the sense that only the basic theory of $\wp
(z|\tau)$ covered by the standard textbook (cf. \cite{A}) are used. For
example, the following lemma only uses some properties of $\wp(z|\tau)$ on rectangles.

\begin{lemma}
\label{11-9-3} Let $\omega_{2}=\tau \in i\mathbb{R}^{+}$. Then $\wp$ is one to
one from $(0,\frac{1}{2}\omega_{1}]\cup \lbrack \frac{1}{2}\omega_{1},\frac
{1}{2}\omega_{3}]\cup \lbrack \frac{1}{2}\omega_{3},\frac{1}{2}\omega_{2}%
]\cup \lbrack \frac{1}{2}\omega_{2},0)$ onto $(-\infty,+\infty)$. Here
$[z_{1},z_{2}]=\{z:z=tz_{2}+(1-t)z_{1},\;0\leq t\leq1\}$.
\end{lemma}

\begin{proof}
By $\tau \in i\mathbb{R}^{+}$ and the definition of $\wp(z)$:
\begin{equation}
\wp(z)=\frac{1}{z^{2}}+\sum_{(m,n)\neq(0,0)}\Bigl(\frac{1}{(z-m-n\tau)^{2}%
}-\frac{1}{(m+n\tau)^{2}}\Bigr), \label{1-9-3}%
\end{equation}
it is easy to see that $\overline{\wp(z)}=\wp(\bar{z})$. Since $\bar{z}=z$ if
$z\in(0,\frac{1}{2}]$, $\bar{z}=-z$ if $z\in(0,\frac{\tau}{2}]$, $\bar{z}=1-z$
if $z\in \lbrack \frac{1}{2},\frac{1+\tau}{2}]$, $\bar{z}=z-\tau$ if
$z\in \lbrack \frac{1+\tau}{2},\frac{\tau}{2}]$, so $\wp$ is real-valued in
$(0,\frac{\tau}{2}]\cup \lbrack \frac{\tau}{2},\frac{1+\tau}{2}]\cup \lbrack
\frac{1+\tau}{2},\frac{1}{2}]\cup \lbrack \frac{1}{2},0)$.

On the other hand, since $\wp(z)=\wp(-z)$ and the degree of $\wp(z)$ is two,
we conclude that $\wp(z)$ is one to one in $(0,\frac{\tau}{2}]\cup \lbrack
\frac{\tau}{2},\frac{1+\tau}{2}]\cup \lbrack \frac{1+\tau}{2},\frac{1}{2}%
]\cup \lbrack \frac{1}{2},0)$. Moreover, since the second term in the RHS of \eqref{1-9-3} is bounded as $z\rightarrow0$, we conclude
\[
\lim_{\,[\frac{1}{2},0)\ni z\rightarrow0}\wp(z)=+\infty,\quad \lim
_{(0,\frac{\tau}{2}]\ni z\rightarrow0}\wp(z)=-\infty.
\]
The proof is complete.
\end{proof}

\begin{remark}
\label{re1-10-3} Let $e_k=\wp(\frac{\omega_k}{2})$, $k=1,2,3$. We recall that $\wp(z)$ satisfies the cubic equation:
{\allowdisplaybreaks
\begin{align}
\label{new34} &\wp'(z)^2 = 4\wp(z)^3 - g_2\wp(z) - g_3=4\prod_{k=1}^3(\wp(z)-e_k), \\  & \text{and}\;\;\wp''(z) = 6\wp(z)^2 - g_2/2. \nonumber\end{align}}%
Thus $e_{1} + e_{2} + e_{3}=0$. From Lemma \ref{11-9-3}, we have $e_{j}%
\in \mathbb{R}$, $e_{2}<e_{3}<e_{1}$ and $e_{2}<0<e_{1}$, also $\wp^{\prime
}(z)=\frac{\partial \wp(z)}{\partial x_{1}}\in \mathbb{R}$ if $z\in(0,\frac
{1}{2}\omega_{1}]\cup(\frac{1}{2}\omega_{2},\frac{1}{2}\omega_{3}]$, and
$\wp^{\prime}(z)=-i\frac{\partial \wp(z)}{\partial x_{2}}\in i\mathbb{R}$ if
$z\in(0,\frac{1}{2}\omega_{2}]\cup(\frac{1}{2}\omega_{1},\frac{1}{2}\omega
_{3}]$. Lemma~\ref{11-9-3} also implies $\zeta(z)\in \mathbb{R}$ for
$z\in(0,\frac{1}{2}\omega_{1}]$ and so $\eta_{1}\in \mathbb{R}$. In the following, we use $q_{\pm}$ to denote the solution of $\wp(q_{\pm})=\pm\sqrt{g_2/12}$, i.e. $\wp''(q_{\pm})=0$.
\end{remark}

Recall our assumption that $E_{\tau}$ is a rectangle centered at the
origin. We first discuss (\ref{new213}) by assuming $a\in \partial E_{\tau}$. To prove
Theorem 1.1 in this case, we will solve the second equation in (\ref{new213})
to obtain a branch $b=b(a)$, and then insert $b=b(a)$ in the first equation of
(\ref{new213}) to find a contradiction. For this purpose, we now discuss the
second equation in (\ref{new213}) with $a\neq-b$. We have the following lemma.

\begin{lemma}
\label{l1-30-1}
The equation $\wp^{\prime \prime}(a)=0$ has exactly four distinct
solutions $\pm q_{\pm}$, which all belong to $\partial E_{\tau}$ with
$q_{+}\in(\frac{1}{2}\omega_{1},\frac{1}{2}\omega_{3})$ and $q_{-}\in(\frac
{1}{2}\omega_{2},\frac{1}{2}\omega_{3})$. Moreover, for any $a\in E_{\tau
}\setminus \{ \pm q_{\pm},\pm2q_{\pm}\}$, there are two distinct solutions $b$'s
to the equation
\[
\wp^{\prime}(a)+\wp^{\prime}(b)=0,\quad a\neq-b.
\]
\end{lemma}

\begin{proof}
From (\ref{new34}), $\wp^{\prime \prime}(z)=0$ has $4$ zeros at $\pm q_{+}, \pm q_{-}$,
where $\wp(q_{\pm})=\pm \sqrt{g_{2}/12}$, and
\begin{equation}
e_{1}+e_{2}+e_{3}=0,\quad e_{1}e_{2}+e_{1}e_{3}+e_{2}e_{3}=-\frac{g_{2}}%
{4},\quad e_{1}e_{2}e_{3}=\frac{g_{3}}{4}, \label{nn1-3-2}%
\end{equation}
which implies $g_{2}=2(e_{1}^{2}+e_{2}^{2}+e_{3}^{2})>0$. So $\wp(q_{\pm}%
)\in \mathbb{R}$. We claim
\begin{equation}
e_{2}<-\sqrt{g_{2}/12}<e_{3}<\sqrt{g_{2}/12}<e_{1}.
\label{fc1-1}%
\end{equation}
Then it follows that $q_{+}\in(\frac{1}{2}\omega_{1},\frac{1}{2}\omega_{3})$
and $q_{-}\in(\frac{1}{2}\omega_{2},\frac{1}{2}\omega_{3})$, i.e. $\pm
q_{\pm}\in \partial E_{\tau}$.

Since $e_{1}+e_{2}+e_{3}=0$ and $e_{2}<e_{3}<e_{1}$ by Remark \ref{re1-10-3},
we have $e_{2}<0$, $e_{1}>0$ and $|e_{3}|<\min \{|e_{2}|, e_{1}\}$. Thus, for
$i=1$ or $i=2$,
\[
g_{2}=2(e_{1}^{2}+e_{2}^{2}+e_{3}^{2})=4(e_{i}^{2}+e_{i}e_{3}+e_{3}%
^{2})<12e_{i}^{2},
\]
namely $e_{2}<-\sqrt{g_{2}/12}$ and $e_{1}>\sqrt{g_{2}/12}$.
If $e_{3}\le0$, then $g_{2}=4(e_{2}^{2}+e_{2}e_{3}+e_{3}^{2})> 12e_{3}^{2}$;
if $e_{3}>0$, then $g_{2}=4(e_{1}^{2}+e_{1}e_{3}+e_{3}^{2})> 12e_{3}^{2}$.
Therefore, $|e_{3}|<\sqrt{g_{2}/12}$, namely (\ref{fc1-1}) holds.

For any $a\in E_{\tau}$, $\wp^{\prime}(z)=-\wp^{\prime}(a)$ has three
solutions, because the degree of the map $\wp^{\prime}$ from $E_{\tau}$ to
$\mathbb{C}\cup \{ \infty \}$ is three. Note that $\wp^{\prime \prime}(z)=0$ if
and only if $z=\pm q_{\pm}$. Thus $\wp^{\prime}(a)+\wp^{\prime}(b)=0$ has
three distinct solutions $b$'s except for those $a$'s such that $\wp^{\prime
}(a)+\wp^{\prime}(\pm q_{\pm})=0$ for some $\pm q_{\pm}$. To find such $a$, we
note that%
\[
\wp^{\prime}(a)^{2}=\wp^{\prime}(b)^{2},\quad \text{for some $b\in \{ \pm
q_{+}, \pm q_{-}\}$}.
\]
It suffices to consider the case $a\notin \{ \pm q_{\pm}\}$. Then $\wp
(a)\neq \wp(b)$. By using%
\begin{equation}
\wp^{\prime}(z)^{2}=4\wp(z)^{3}-g_{2}\wp(z)-g_{3} \label{1-3-2}%
\end{equation}
at $z=a$ and $z=b$, we have
\begin{equation}
\wp(a)^{2}+\wp(a)\wp(b)+\wp(b)^{2}-\frac{g_{2}}{4}=0. \label{10-22-1}%
\end{equation}
Recalling $\wp(b)=\pm \sqrt{g_{2}/12}$ for $b\in \{ \pm q_{\pm}\}$, we
get
\begin{equation}
\wp(a)=\frac{-\wp(b)\pm \sqrt{g_{2}-3\wp(b)^{2}}}{2}=\frac{-\wp(b)\pm3\wp
(b)}{2}. \label{fc1-2}%
\end{equation}
This, together with $\wp(a)\neq \wp(b)$, gives $\wp(a)=-2\wp(b)$. From the
addition formula $\wp(2z)=\frac{1}{4}(\frac{\wp^{\prime \prime}(z)}{\wp
^{\prime}(z)})^{2}-2\wp(z)$ and $\wp^{\prime \prime}(b)=0$ for $b\in \{ \pm
q_{\pm}\}$, we get $\wp(a)=\wp(2b)$. Therefore, $a\in \{ \pm2q_{\pm}\}$. This
completes the proof.
\end{proof}

\begin{remark}
\label{11-11-1} We have proved $q_{+}\in(\frac{1}{2}\omega_{1},\frac{1}{2}\omega_{3})$ and $q_{-}\in(\frac
{1}{2}\omega_{2},\frac{1}{2}\omega_{3})$. From $e_{1}+e_{2}+e_{3}=0$, we have $g_{2}=2(e_{1}^{2}%
+e_{2}^{2}+e_{3}^{2}) >3\max\{e^{2}_{1},e^{2}_{2}\}$, which implies $\wp
(2q_{+})=-2\wp(q_{+})=-\sqrt{g_2/3} < e_{2}$ and $\wp(2q_{-})=-2\wp(q_{-})=\sqrt{g_2/3}>e_{1}$. Hence
$2q_{+} \in(0,\frac{\omega_{2}}{2})\cup(-\frac{\omega_{2}}{2},0)$ and $2q_{-} \in(0,\frac{\omega_{1}}{2})\cup(-\frac{\omega_{1}}{2},0)$. We will prove in Lemma \ref{lemma4-7} that $2q_{+} \in(0,\frac{\omega_{2}}{2})$.
\end{remark}

\begin{lemma}
\label{l1-22-1} There is no pair $(a,b)$ with $a$ or $b\in \partial E_{\tau}$,
such that (\ref{new213}) holds.
\end{lemma}

\begin{proof}
Assume by contradiction that such $(a,b)$ exists. Since the degree of $\wp(z)$
is two and $\wp(-z)=\wp(z)$, we know that $\wp(a)\neq \wp(b)$ because of
$a\neq\pm b$. Then just as in Lemma \ref{l1-30-1}, it follows from $\wp^{\prime}%
(a)+\wp^{\prime}(b)=0$ that \eqref{10-22-1} holds for
$(\wp(a),\wp(b))$.

Without loss of generality, we assume $a\in \partial E_{\tau}$. From
\eqref{10-22-1}, we find%
\begin{equation}
\wp(b)=\frac{-\wp(a)\pm \sqrt{g_{2}-3\wp(a)^{2}}}{2}. \label{fc1-5}%
\end{equation}
We claim
\begin{equation}
\label{eq310}g_{2}-3\wp(a)^{2}>0\  \  \text{for any}\ a\in \partial E_{\tau}.
\end{equation}
From $\wp(-z)=\wp(z)$ and $\wp(z+\omega_{j})=\wp(z)$, $j=1,2$, we only need to
prove the claim for $a\in \lbrack \frac{1}{2}\omega_{2},\frac{1}{2}\omega
_{3}]\cup \lbrack \frac{1}{2}\omega_{1},\frac{1}{2}\omega_{3}]$. Let us assume
$a\in \lbrack \frac{1}{2}\omega_{1},\frac{1}{2}\omega_{3}]$. Then $e_{3}\leq
\wp(a)\leq e_{1}$. If $\wp(a)\leq0$, then from \eqref{nn1-3-2} and $e_{1}%
e_{2}<0$, we have
\begin{equation}
g_{2}=-4(e_{1}e_{2}+e_{3}(e_{1}+e_{2}))=4(e_{3}^{2}-e_{1}e_{2})>4e_{3}%
^{2}>3\wp(a)^{2}. \label{10-10-3}%
\end{equation}
On the other hand, if $\wp(a)>0$, by $e_{1}^{2}-4e_{2}e_{3}=(e_{2}-e_{3}%
)^{2}>0$, we have
\begin{equation}
g_{2}=4(e_{1}^{2}-e_{2}e_{3})>3e_{1}^{2}\geq3\wp(a)^{2}. \label{11-10-3}%
\end{equation}
Suppose now $a\in \lbrack \frac{1}{2}\omega_{2},\frac{1}{2}\omega_{3}]$. Then
$e_{2}\le\wp(a)\le e_{3}$. If $\wp(a)>0$, then \eqref{10-10-3} gives $g_{2}%
>3\wp(a)^{2}$. If $\wp(a)\leq0$, then similar to \eqref{11-10-3}, we have
\[
g_{2}=4(e_{2}^{2}-e_{1}e_{3})>3e_{2}^{2}\geq3\wp(a)^{2}.
\]
So, the claim (\ref{eq310}) follows. Since $\wp(a)\in \mathbb{R}$, by the claim and (\ref{fc1-5})
we also have $\wp(b)\in \mathbb{R}$.

To prove Lemma 3.7, let us argue for the case $a\in \bigl(\frac{1}{2}%
(\omega_{1}-\omega_{2}),\frac{1}{2}\omega_{1}\bigr)\cup \bigl(\frac{1}{2}%
\omega_{1},\frac{1}{2}\omega_{3}\bigr)$, which are two intervals on the line
$\frac{1}{2}\omega_{1}+i\mathbb{R}$. Without loss of generality, we may assume
$a\in \lbrack \frac{1}{2}\omega_{1},\frac{1}{2}\omega_{3}]$. Lemma~\ref{l1-30-1}
and Remark 3.6 tell us that there are three branch solutions $b_{i}(a)$,
$i=1,2,3$, of $\wp^{\prime}(a)+\wp^{\prime}(b)=0$ for $a\in \lbrack \frac{1}%
{2}\omega_{1},\frac{1}{2}\omega_{3}]\backslash \{q_{+}\}$, where we assign
$b_{1}(a)=-a$ for any $a$. To continue our proof, we need two lemmas to study
the basic properties of the other two branches.

\begin{lemma}
\label{lemma4-7} For $a\in[\frac{1}{2}w_{1}, \frac{1}{2}w_{3}]$, there are two
analytic branches $b_{2}(a)$ and $b_{3}(a)$ of solutions to $\wp^{\prime
}(a)+\wp^{\prime}(b)=0$ such that $b_{2}(\frac12 \omega_{1})=-\frac12
\omega_{3}$, $b_{2}(q_{+})=-q_{+}$ and $b_{2}(\frac12 \omega_{3})=-\frac12
\omega_{1}$, $b_{3}(\frac{1}{2}\omega_{1})=\frac{1}2 \omega_{2}$, $b_{3}%
(q_{+})=2q_{+}$ and $b_{3}(\frac12 \omega_{3})=\frac12 \omega_{2}$.
Furthermore, $b_{2}(a)\in[-\frac{1}{2}\omega_{3}, -\frac12 \omega_{1}]$ and
$b_{3}(a)\in[2q_{+}, \frac12 \omega_{2}]$, $2q_{+}\in(0, \frac12\omega_{2})$.
\end{lemma}

\begin{proof}
For $a\in \lbrack \frac{1}{2}\omega_{1},q_{+})$, there exist two analytic branch
solutions $b_{2}(a)$ and $b_{3}(a)$ for $\wp^{\prime}(a)+\wp^{\prime}(b)=0$.
Since $\wp^{\prime}(\frac{1}{2}\omega_{1})=0$, we have $\wp^{\prime}%
(b(\frac{1}{2}\omega_{1}))=0$. Hence, $b(\frac{1}{2}\omega_{1})=\frac{1}%
{2}\omega_{2}$ or $b(\frac{1}{2}\omega_{1})=-\frac{1}{2}\omega_{3}$ since
$a\neq \pm b$. Here, we assume $b_{2}(\frac{1}{2}\omega_{1})=-\frac{1}{2}%
\omega_{3}$ and $b_{3}(\frac{1}{2}\omega_{1})=\frac{1}{2}\omega_{2}$. By
Lemma~\ref{11-9-3}, $\wp(a)$ is decreasing in $[\frac{1}{2}\omega_{1},\frac
{1}{2}\omega_{3}]$, we find $\wp^{\prime}(a)\in i\mathbb{R}^{+}$ for
$a\in(\frac{1}{2}\omega_{1},\frac{1}{2}\omega_{3})$, which gives $\wp^{\prime
}(b_{i}(a))\in i\mathbb{R}^{-}$. By \eqref{fc1-5}, $g_{2}-3\wp(a)^{2}>0$ and
Lemma~\ref{11-9-3}, we find $\wp(b_{i}(a))\in \mathbb{R}$. Together with Remark \ref{re1-10-3}, we conclude that
\[
b_{2}:[\tfrac{1}{2}\omega_{1},q_{+})\rightarrow-\tfrac{1}{2}\omega
_{3}+i\mathbb{R}^{+},
\]
and
\[
b_{3}:[\tfrac{1}{2}\omega_{1},q_{+})\rightarrow \lbrack \tfrac{1}{2}\omega
_{2},0).
\]

First, we note that $b_{2}$ is one-to-one for $a\in \lbrack \frac{1}{2}%
\omega_{1},q_{+})$, because if $b_{2}(a)=b_{2}(\tilde{a})$ for some
$a,\tilde{a}\in \lbrack \frac{1}{2}\omega_{1},q_{+})$, then $\wp^{\prime
}(a)=-\wp^{\prime}(b_{2}(a))=-\wp^{\prime}(b_{2}(\tilde{a}))=\wp^{\prime
}(\tilde{a})$, which implies $a=\tilde{a}$, since $\wp^{\prime \prime}\neq0$ on
$[\frac{1}{2}\omega_{1},q_{+})$. Similarly, $b_{3}$ is one-to-one for
$a\in \lbrack \frac{1}{2}\omega_{1},q_{+})$. By one-to-one, $b_{2}(a)$ is
increasing from $b_{2}(\frac{1}{2}\omega_{1})=-\frac{1}{2}\omega_{3}$ to
$b_{2}(q_{+})=\lim_{a\rightarrow q_{+}}b_{2}(a)$ as $a$ varies from $\frac
{1}{2}\omega_{1}$ to $q_{+}$. The previous proof of Lemma \ref{l1-22-1} shows
that (\ref{10-22-1}) holds for $a\in \lbrack \frac{1}{2}\omega_{1},q_{+})$ and
$b_{2}(a)$. By letting $a\rightarrow q_{+}$, we also have that $(q_{+}%
,b_{2}(q_{+}))$ satisfies (\ref{10-22-1}). Then similarly to (\ref{fc1-2}), we
obtain
\[
\wp(b_{2}(q_{+}))=\frac{-\wp(q_{+})\pm3\wp(q_{+})}{2},
\]
namely either $\wp(b_{2}(q_{+}))=\wp(q_{+})$ or $\wp(b_{2}(q_{+}))=-2\wp(q_{+}%
)=\wp(2q_{+})$ because $\wp^{\prime \prime}(q_{+})=0$. Since $b_{2}(q_{+}%
)\in-\frac{1}{2}\omega_{3}+i\mathbb{R}^{+}$ and $2q_{+}\in \omega
_{1}+i\mathbb{R}=i\mathbb{R}$ in the torus $E_{\tau}$, we conclude that
$b_{2}(q_{+})=-q_{+}$.

The above argument also shows $\wp(b_{3}(q_{+}))=-2\wp(q_{+})=\wp(2q_{+})$. So
we have either $b_{3}(q_{+})=2q_{+}$ or $b_{3}(q_{+})=-2q_{+}$. We claim
\begin{equation}
\label{fc1-3}b_{3}(q_{+})=2q_{+}.
\end{equation}

Recalling $b_3(\frac{1}{2}\omega_{1})=\frac{1}{2}\omega_{2}$, (\ref{fc1-3}) is equivalent to $2q_{+}\in(0,\frac{1}{2}\omega_{2})$. So it
suffices to prove $q_{+}\in(\frac{1}{2}\omega_{1},\frac{1}{2}\omega_{1}%
+\frac{1}{4}\omega_{2})$ or equivalently, to show $\wp(q_{+})>\wp(\frac{1}%
{2}\omega_{1}+\frac{1}{4}\omega_{2})$.
We use the following addition formula
to prove this inequality:%
\begin{equation}
\wp(2z)+2\wp(z)=\frac{1}{4}\left(  \frac{\wp^{\prime \prime}(z)}{\wp^{\prime
}(z)}\right)  ^{2}. \label{fc1-4}%
\end{equation}
Because $0\neq \wp^{\prime}(\frac{1}{2}\omega_{1}+\frac{1}{4}\omega_{2})\in
i\mathbb{R}$ and $\wp^{\prime \prime}(\frac{1}{2}\omega_{1}+\frac{1}{4}%
\omega_{2})\in \mathbb{R}$, (\ref{fc1-4}) gives
\[
2\wp(\tfrac{1}{2}\omega_{1}+\tfrac{1}{4}\omega_{2})\leq-\wp(\tfrac{1}{2}%
\omega_{2})=-e_{2}<2\wp(q_{+}),
\]
where the last inequality follows from Remark \ref{11-11-1}. Hence (\ref{fc1-3}) is proved.

It is easy to see that these two branches $b_{2}(a)$ and $b_{3}(a)$ can be
extended from $[\frac{1}{2}\omega_{1},q_{+}]$ to $[\frac{1}{2}\omega_{1}%
,\frac{1}{2}\omega_{3}]$ such that for $a\in(q_{+},\frac{1}{2}\omega_{3}]$,
$b_{2}(a)\in(-q_{+},-\frac{1}{2}\omega_{1}]$ and $b_{3}(a)\in(2q_{+},\frac
{1}{2}\omega_{2}]$. This completes the proof.
\end{proof}

\begin{lemma}
\label{lemma4-8} For $a\in[\frac12 \omega_{1}, \frac12 \omega_{3}]$, the
following statements hold:

\begin{itemize}
\item[$(i)$] { $-e_{2}\le \wp(a)+\wp(b_{2}(a))\le2\wp(q_{+});$ }

\item[$(ii)$] { $-e_{1}\le \wp(a)+\wp(b_{3}(a))\le-e_{3}$. }
\end{itemize}
\end{lemma}

\begin{proof}
We define $f_{i}(a):=\wp(a)+\wp(b_{i}(a))$, $i=2, 3$. Then for $a\in(\frac12
\omega_{1}, \frac12 \omega_{3})$,
\[
f_{i}^{\prime}(a)=\wp^{\prime}(a)+\wp^{\prime}(b_{i}(a))b_{i}^{\prime}%
(a)=\wp^{\prime}(a)(1-b_{i}^{\prime}(a)).
\]
Note that $\wp^{\prime}(a)\neq0$ for $a\in(\frac12 \omega_{1}, \frac12
\omega_{3})$. Assume that $\bar{a}\in(\frac12 \omega_{1}, \frac12 \omega_{3})$
is a critical point of $f_{i}$. Then $b_{i}^{\prime}(\bar a)=1$. By the
arguments in the proof of Lemma \ref{lemma4-7}, we know that (\ref{10-22-1})
holds for $(\wp(a), \wp(b_{i}(a)))$. Differentiating over (\ref{10-22-1}), we
easily conclude that
\[
[\wp(a)+2\wp(b_{i}(a))]b_{i}^{\prime}(a)=2\wp(a)+\wp(b_{i}(a)).
\]
Recalling (\ref{fc1-5}), we have $\wp(a)+2\wp(b_{i}(a))=\pm \sqrt{g_{2}%
-3\wp(a)^{2}}\neq0$. Thus,
\begin{equation}
\label{fc1-6}b_{i}^{\prime}(a)=\frac{2\wp(a)+\wp(b_{i}(a))}{\wp(a)+2\wp
(b_{i}(a))}.
\end{equation}
Letting $a=\bar a$ in (\ref{fc1-6}), we obtain $\wp(b_{i}(\bar a))=\wp(\bar
a)$. This, together with (\ref{fc1-5}), gives
\[
\wp(\bar a)=\wp(b_{i}(\bar a))=\frac{-\wp(\bar a)\pm \sqrt{g_{2}- 3\wp^{2}(\bar
a) } }2,
\]
which implies $\wp(b_{i}(\bar a))=\wp(\bar a)=\pm \sqrt{g_{2}/12}$.
Thus, $\bar a=q_{+}$ and so $b_{i}(q_{+})=-q_{+}$. Therefore, $q_{+}$ is the
only critical point of $f_{2}$ in $(\frac12 \omega_{1}, \frac12 \omega_{3})$,
while $f_{3}$ has no critical points in $(\frac12 \omega_{1}, \frac12
\omega_{3})$, namely $f_{3}$ is strictly monotone in $[\frac12 \omega_{1},
\frac12 \omega_{3}]$. By Lemma \ref{lemma4-7}, $f_{2}(\frac12 \omega
_{1})=f_{2}(\frac12 \omega_{3})=e_{1}+e_{3}=-e_{2}<\sqrt{g_{2}/3}%
=2\wp(q_{+})=f_{2}(q_{+})$, hence $(i)$ holds. Besides, $f_{3}(\frac12
\omega_{1})=e_{1}+e_{2}=-e_{3}$ and $f_{3}(\frac12\omega_{3})=e_{3}%
+e_{2}=-e_{1}$, we see that $(ii)$ holds.
\end{proof}

Now we go back to the proof of Lemma \ref{l1-22-1}. First let us consider
$b_{2}(a)$. Since $b_{2}(q_{+})=-q_{+}$, $\nabla G(q_{+})+\nabla G(b_{2}%
(q_{+}))=0$ due to the anti-symmetry of $\nabla G$. We will show that $\nabla
G(a)+\nabla G(b_{2}(a))\neq0$ for all $a\in(\frac{1}{2}\omega_{1},\frac{1}%
{2}\omega_{3})\setminus \{q_{+}\}$. For this purpose, we consider the following
real-valued function on $a\in I=[\frac{1}{2}\omega_{1},\frac{1}{2}\omega_{3}%
]$:
\[
H_{2}(a):=G_{x_{2}}(a)+G_{x_{2}}(b_{2}(a)).
\]
Since $b_{2}(\frac{1}{2}\omega_{1})=-\frac{1}{2}\omega_{3}$ and $b_{2}%
(\frac{1}{2}\omega_{3})=-\frac{1}{2}\omega_{1}$, we have $H_{2}(a)=0$ if
$a\in \{ \frac{1}{2}\omega_{3},\frac{1}{2}\omega_{1},q_{+}\}$. We want to show
that there is no other zeros of $H_{2}(a)=0$ in $[\frac{1}{2}\omega_{1}%
,\frac{1}{2}\omega_{3}]$. Note that $H_{2}^{\prime}(a)=0$ has at least two
solutions because $H_{2}(a)=0$ at $\frac{1}{2}\omega_{1}$, $\frac{1}{2}%
\omega_{3}$ and $q_{+}$. If we can prove that $H_{2}^{\prime}(a)=0$ has only
two solutions in $\bigl(\frac{1}{2}\omega_{1},\frac{1}{2}\omega_{3}\bigr)$,
then except the three points $\frac{1}{2}\omega_{1}$, $\frac{1}{2}\omega_{3}$
and $q_{+}$, $H_{2}(a)$ has no other zeros in $[\frac{1}{2}\omega_{1},\frac
{1}{2}\omega_{3}]$.

Let us compute $H_{2}^{\prime}(a)$. Note that $G_{x_{2}x_{2}}(a)$ and
$G_{x_{2}x_{2}}(b_{2}(a))$ can be derived as follows. From (\ref{40}), we have%
\[
\left(4\pi G_{z}(z)+\frac{2\pi ix_{2}}{\operatorname{Im}\tau}\right)^{\prime
}=\bigl(-\zeta(z)+\eta_{1}z\bigr)^{\prime}=\wp(z)+\eta_{1}.
\]
But%
\begin{align*}
&  \left(4\pi G_{z}(z)+\frac{2\pi ix_{2}}{\operatorname{Im}\tau}%
\right)^{\prime}=\frac{\partial}{\partial x_{1}}\left(4\pi G_{z}(z)+\frac{2\pi
ix_{2}}{\operatorname{Im}\tau}\right)=4\pi \frac{\partial G_{z}(z)}{\partial
x_{1}}\\
=  &  2\pi G_{x_{1}x_{1}}-2\pi iG_{x_{1}x_{2}}=-2\pi G_{x_{2}x_{2}}-2\pi
iG_{x_{1}x_{2}}+\frac{2\pi}{\operatorname{Im}\tau}.
\end{align*}
Thus we obtain
\begin{equation}
2\pi G_{x_{1}x_{1}}(z)=\operatorname{Re}(\eta_{1}+\wp(z)), \label{201-10-3}%
\end{equation}%
\begin{equation}
2\pi G_{x_{1}x_{2}}(z)=-\operatorname{Im}(\eta_{1}+\wp(z)), \label{201-10-4}%
\end{equation}%
\begin{equation}
2\pi G_{x_{2}x_{2}}(z)=\frac{2\pi}{\operatorname{Im}\tau}-\operatorname{Re}%
(\eta_{1}+\wp(z)). \label{201-10-5}%
\end{equation}
Since $\wp(z)$ is real for $z=a$ or $b_{2}(a)$, we have%
\begin{align}
2\pi iH_{2}^{\prime}(a)  &  =2\pi G_{x_{2}x_{2}}(a)+2\pi G_{x_{2}x_{2}}%
(b_{2}(a))b_{2}^{\prime}(a)\nonumber \\
&  =\frac{2\pi}{\operatorname{Im}\tau}-\eta_{1}-\wp(a)+\left(  \frac{2\pi
}{\operatorname{Im}\tau}-\eta_{1}-\wp(b_{2}(a))\right)  b_{2}^{\prime}(a).
\label{11-22-1}%
\end{align}
For $a\in \frac{1}{2}\omega_{1}+i\mathbb{R}$, $H_{2}^{\prime}(a)\in
i\mathbb{R}$. Recalling (\ref{fc1-6}) and denoting $\tilde{\eta}_{1}=\eta
_{1}-\frac{2\pi}{\operatorname{Im}\tau}$ for convenience, we see that
$H_{2}^{\prime}(a)=0$ is equivalent to
\[
\tilde{\eta}_{1}+\wp(a)+\left(  \tilde{\eta}_{1}+\wp(b_{2}(a))\right)
\frac{2\wp(a)+\wp(b_{2}(a))}{\wp(a)+2\wp(b_{2}(a))}=0.
\]
By direct computations, we get
\begin{equation}
3\tilde{\eta}_{1}(\wp(a)+\wp(b_{2}(a)))+2\wp(a)\wp(b_{2}(a))+\left[
\wp(a)+\wp(b_{2}(a))\right]  ^{2}=0. \label{12-22-1}%
\end{equation}
By \eqref{10-22-1}, $\wp(a)\wp(b_{2}(a))=[\wp(a)+\wp(b_{2}(a))]^{2}-g_{2}/4$.
Insert this into \eqref{12-22-1}, we obtain
\[
\lbrack \wp(a)+\wp(b_{2}(a))]^{2}+\tilde{\eta}_{1}(\wp(a)+\wp(b_{2}%
(a)))-\frac{g_{2}}{6}=0.
\]
Thus,
\begin{equation}
f_{2}(a)=\wp(a)+\wp(b_{2}(a))=\frac{1}{2}\left(  -\tilde{\eta}_{1}\pm
\sqrt{\tilde{\eta}_{1}^{2}+2g_{2}/3}\right)  :=B_{\pm}. \label{fc1-7}%
\end{equation}
Clearly $B_{+}>0>B_{-}$. By Lemma \ref{lemma4-8}, we have
\begin{equation}
f_{2}(a)=B_{+}\geq-e_{2}>0. \label{fc4-33}%
\end{equation}
Combining this with (\ref{12-22-1}), we conclude that
\[
\wp(a)+\wp(b_{2}(a))=B_{+}\quad \text{and}\quad \wp(a)\wp(b_{2}(a))=-\frac
{B_{+}^{2}+3\tilde{\eta}_{1}B_{+}}{2}=:A_{+},
\]
and so
\[
\wp(a)=\frac{B_{+}\pm \sqrt{B_{+}^{2}-4A_{+}}}{2},
\]
whenever $H_{2}^{\prime}(a)=0$. Since $\wp$ is one-to-one on $[\frac{1}%
{2}\omega_{1},\frac{1}{2}\omega_{3}]$, there are two distinct points $a_{+}$
and $a_{-}$ such that $\wp(a_{\pm})=\frac{B_{+}\pm \sqrt{B_{+}^{2}-4A_{+}}}{2}%
$. Hence, we have proved that $H_{2}^{\prime}(a)=0$ has exactly two zero
points in $(\frac{1}{2}\omega_{1},q_{+})\cup(q_{+},\frac{1}{2}\omega_{3})$,
which implies that $H(a)\neq0$ for any $a\in(\frac{1}{2}\omega_{1},q_{+}%
)\cup(q_{+},\frac{1}{2}\omega_{3})$. In conclusion, for any $a\in [\frac{1}{2}\omega_{1},\frac{1}{2}\omega_{3}]$, $(a, b_2(a))$ can not satisfy (\ref{new213}).

Next we consider $b_{3}(a)$. We also define
\[
H_{3}(a)=G_{x_{2}}(a)+G_{x_{2}}(b_{3}(a)).
\]

The difference is that $H_{3}(q_{+})\neq0$ since $b_{3}(q_{+})=2q_{+}\in
(0,\frac{1}{2}\omega_2)\subset i\mathbb{R}^+$. Thus, we have to show that $H_{3}(a)$ has only two zeros at
$\frac{1}{2}\omega_{1}$ and $\frac{1}{2}\omega_{3}$, namely we need to prove
$H_{3}^{\prime}(a)=0$ has only one zero point. The computation of
$H_{3}^{\prime}(a)$ is completely the same as $H_{2}^{\prime}(a)$. Hence,
$H_{3}^{\prime}(a)=0$ implies (see \eqref{fc1-7})
\[
f_{3}(a)=\wp(a)+\wp(b_{3}(a))=\frac{1}{2}\left(  -\tilde{\eta}_{1}\pm
\sqrt{\tilde{\eta}_{1}^{2}+2g_{2}/3}\right)  =B_{\pm}.
\]
We note that this $B_{\pm}$ is the same one in (\ref{fc1-7}). Recall from
Lemma \ref{lemma4-8} that $f_{3}$ is strict monotone in $[\frac{1}{2}%
\omega_{1},\frac{1}{2}\omega_{3}]$ and $-e_{1}\leq f_{3}\leq-e_{3}$. Since
$B_{+}\geq-e_{2}>-e_{3}$ by (\ref{fc4-33}), it follows that $f_{3}(a)=B_{-}$
whenever $H_{3}^{\prime}(a)=0$. By the monotonicity of $f_{3}$, the $a$
satisfying $f_{3}(a)=B_{-}$ is unique. Thus $H_{3}^{\prime}(a)=0$ has only one
solution in $[\frac{1}{2}\omega_{1},\frac{1}{2}\omega_{3}]$, and then
$H_{3}(a)\neq0$ for any $a\in(\frac{1}{2}\omega_{1},\frac{1}{2}\omega_{3})$. In conclusion, for any $a\in [\frac{1}{2}\omega_{1},\frac{1}{2}\omega_{3}]$, $(a, b_3(a))$ can not satisfy (\ref{new213}).
This completes the proof of Lemma \ref{l1-22-1}.
\end{proof}

\begin{lemma}
\label{l2-22-1} Let $(a, b)$ be a solution of (\ref{new213}).
Then neither $a$ nor $b$ can be on the coordinate axes.
\end{lemma}

\begin{proof}
Suppose that $a$ is on the $x_{1}$ axis. Note from $a\neq -a$ that $a\notin\{ 0, \pm\frac{\omega_{1}}{2}\}$.
Without loss of generality, we may assume $a>0$, i.e. $a\in (0,\frac{\omega_1}{2})$. Then (cf. \cite[Lemma
2.1]{CLW4})%
\[
G_{x_{1}}(a)<0,\quad G_{x_{2}}(a)=0.
\]
As a result, $G_{x_{2}}(b)=-G_{x_{2}}(a)=0$. It is known (cf. \cite[Lemma
2.1]{CLW4}) that $G$ satisfies%
\[
G_{x_{2}}(z)\neq0,\quad \text{if}\;z\in E_\tau\setminus(\mathbb{R}\cup \bigl(\pm\tfrac{1}%
{2}\omega_{2}+\mathbb{R}\bigr)).
\]
So $b\in \mathbb{R}\cup \bigl(\pm\frac{1}{2}\omega_{2}+\mathbb{R}\bigr)$. By
Lemma~\ref{l1-22-1}, $b\notin \pm\frac{1}{2}\omega_{2}+\mathbb{R}$. Hence,
$b\in \mathbb{R}$ and $\wp^{\prime}(b)=-\wp^{\prime}(a)>0$. This gives $b\in (-\frac{\omega_1}{2}, 0)=(-\frac{1}{2},0)$.

Note that $\lim_{z\rightarrow0,z<0}\wp^{\prime \prime}(z)=+\infty$ and
$\wp^{\prime \prime}(z)=0$ has solutions only on $\partial E_{\tau}$. So $\wp^{\prime \prime}(x_{1})>0$ for $x_{1}\in (-\frac{1}{2},0)$. This implies that $x_{1}=-a$ is
the only solution of $\wp^{\prime}(x_{1})=-\wp^{\prime}(a)$ for $x_{1}\in (-\frac{1}{2},0)$.
Thus, $b=-a$, a contradiction. The other case that $a$ is on the
$x_{2}$ axis be proved similarly.
\end{proof}

\begin{lemma}
\label{lemma-7} Let $(a_{0},b_{0})$ be a solution of (\ref{new213}). Then
either the $x_{1}$ coordinates or the $x_{2}$ coordinates of $a_{0}$ and
$b_{0}$ take the same sign.
\end{lemma}

\begin{proof}
By Lemma~\ref{l1-22-1}, both $a_{0}$ and $b_{0}\neq \pm a_{0}$ are in the interior
of $E_{\tau}$. Suppose that this lemma fails. Define
\[
T_{t}:=\bigl \{a:\; \;|x_{j}(a)|\leq t|x_{j}(a_{0})|,\; \;j=1,2\bigr \},\quad
t>0.
\]
Here we use $x_{j}(z)$ to denote the $j^{th}$ coordinate of $z$. We say the
sign condition holds for $t$ if for any pair $(a,b(a))$, $a\in T_{t}$, either
$x_{1}(a)x_{1}(b)\geq0$ or $x_{2}(a)x_{2}(b)\geq0$, where $b(a)$ is the branch
of solutions of $\wp^{\prime}(a)+\wp^{\prime}(b)=0$ satisfying $b(a_{0}%
)=b_{0}$. By our assumption, the sign condition fails for $T_{1}$.

On the other hand, if $|z|$ is small, then $\wp^{\prime}(z)=-\frac{2}{z^{3}%
}+O(|z|)$. So if $t$ is small and $a\in T_{t}$, then we can deduce from
$\wp^{\prime}(a)+\wp^{\prime}(b(a))=0$ and $b(a)\neq-a$ that
\[
b(a)=e^{\pm \pi i/3}a(1+O(|a|)).
\]
Thus, the sign condition holds for $T_{t}$ provided that $t$ is small.

Let $t_{0}\in(0,1]$ be the smallest $t$ so that for any small $\varepsilon>0$,
the sign condition fails for $T_{t_{0}+\varepsilon}$. So there is
$a_{\varepsilon}\in T_{t_{0}+\varepsilon}$ such that both $x_{1}%
(a_{\varepsilon})x_{1}(b(a_{\varepsilon}))<0$ and $x_{2}(a_{\varepsilon}%
)x_{2}(b(a_{\varepsilon}))<0$. We may assume $(a_{\varepsilon}, b(a_{\varepsilon}))\to
(\bar{a}_{0}, \bar{b}_{0})$ as $\varepsilon\to 0$ up to a
subsequence. Clearly $\wp^{\prime}(\bar{a}_{0})+\wp^{\prime}(\bar{b}_{0})=0$
and $x_{j}(\bar{a}_{0})x_{j}(\bar{b}_{0})\leq0$ for $j=1,2$. By the choice of
$t_{0}$, $\bar{a}_{0}\in \partial T_{t_{0}}$. Since $\wp^{\prime \prime}(z)=0$
implies $z\in \partial E_{\tau}$, we have $\wp^{\prime \prime}(-\bar{a}_{0}%
)\neq0$, which implies that $-\bar{a}_{0}$ is a simple root of $\wp^{\prime
}(\bar{a}_{0})+\wp^{\prime}(b)=0$. This, together with $b(a_{\varepsilon}%
)\neq-a_{\varepsilon}$, gives $\bar{b}_{0}=b(\bar{a}_{0})\neq-\bar{a}_{0}$.

To yield a contradiction, we first show that one of $\bar{a}_{0}$ or $\bar
{b}_{0}$ must lie on the coordinate axis. If not, then $x_{j}(\bar{a}%
_{0})x_{j}(\bar{b}_{0})<0$ for $j=1,2$. We could choose $a_{\delta}%
:=(1-\delta)\bar{a}_{0}$, $b_{\delta}:=b(a_{\delta})$, such that $(a_{\delta
},b_{\delta})\rightarrow(\bar{a}_{0},\bar{b}_{0})$ as $\delta \rightarrow0$ and
$a_{\delta}\in T_{(1-\delta)t_{0}}$ for $\delta$ small. Clearly, the sign
condition fails for $(a_{\delta},b_{\delta})$ provided $\delta$ is small,
which yields a contradiction to the smallness of $t_{0}$.

Without loss of generality, we assume that one of $\bar{a}_{0}$ and $\bar
{b}_{0}$ is on the imaginary axis. Since $\wp^{\prime}(\bar{a}_{0}%
)+\wp^{\prime}(\bar{b}_{0})=0$, we have both $\wp^{\prime}(\bar{a}_{0})$ and
$\wp^{\prime}(\bar{b}_{0})$ are pure imaginary. Without loss of generality, we
assume $\wp^{\prime}(\bar{a}_{0})=i\xi$ for some real number $\xi>0$. We can
prove the following fact about the curve $\{z:\wp^{\prime}(z)\in
i\mathbb{R}^{+}\}$. For $|z|$ small, $\wp^{\prime}(z)=-\frac{2}{z^{3}%
}+O(|z|)\in i\mathbb{R}^{+}$ if and only if $z=re^{i\theta_{i}}(1+O(r))$,
where $\theta_{i}=\frac{\pi}{6}$, $\frac{5\pi}{6}$, or $\frac{3\pi}{2}$.
Hence, for small $\delta$,%
\begin{equation}
\{|z|\leq \delta \} \cap \{z:\; \wp^{\prime}(z)\in i\mathbb{R}^{+}\} \setminus
i\mathbb{R}^{-}\subset \{z:\;z=(x_{1},x_{2}),x_{2}>0\}. \label{10-7-2}%
\end{equation}
Since $\wp^{\prime}(z)\in \mathbb{R}$ for $z\in \mathbb{R}$, \eqref{10-7-2}
implies%
\begin{equation}
\{z:\; \wp^{\prime}(z)\in i\mathbb{R}^{+}\} \setminus i\mathbb{R}^{-}%
\subset \{z:\;z=(x_{1},x_{2}),x_{2}>0\}. \label{11-7-2}%
\end{equation}
Similarly, we have%
\begin{equation}
\{ \wp^{\prime}(z):\; \wp^{\prime}(z)\in i\mathbb{R}^{-}\} \setminus
i\mathbb{R}^{+}\subset \{z:\;z=(x_{1},x_{2}),x_{2}<0\}. \label{12-7-2}%
\end{equation}

Now we go back to $(\bar{a}_{0},\bar{b}_{0})$. Recall that we have assumed
that one of $\bar{a}_{0}$ and $\bar{b}_{0}$ is on the imaginary axis and $\wp'(\bar{a}_{0})\in i\mathbb{R}^+$. Suppose
that $\bar{a}_{0}$ is on the imaginary axis. Since $\wp(t\omega_{2})$ is
increasing for $t\in(0,\frac{1}{2}]$, we have $\wp^{\prime}(z)\in
i\mathbb{R}^{-}$ for $z\in(0,\frac{1}{2}\omega_{2}]$. By our assumption
$\wp^{\prime}(\bar{a}_{0})\in i\mathbb{R}^{+}$, we find that $\bar{a}_{0}\in
i\mathbb{R}^{-}$. Since $x_{2}(\bar{a}_{0})x_{2}(\bar{b}_{0})\leq0$, we find
$x_{2}(\bar{b}_{0})>0$. From $\wp^{\prime}(\bar{b}_{0})=-\wp^{\prime}(\bar
{a}_{0})\in i\mathbb{R}^{-}$ and \eqref{12-7-2}, we have $\bar{b}_{0}\in
i\mathbb{R}^{+}$. But $\wp^{\prime}(\bar{b}_{0})=-\wp^{\prime}(\bar{a}%
_{0})=\wp^{\prime}(-\bar{a}_{0})$ and both $-\bar{a}_{0}$ and $\bar{b}_{0}$
are on the line $i\mathbb{R}^{+}$, which implies $\bar{b}_{0}=-\bar{a}_{0}$
because $\wp^{\prime \prime}(z)\neq0$ for $z\in(0,\frac{1}{2}\omega_{2})$. This
is a contradiction. Thus we have proved that $\bar{a}_{0}$ is not on the imaginary
axis, which implies that $\bar{b}_{0}$ is on the imaginary axis. Since
$\wp^{\prime}(\bar{b}_{0})\in i\mathbb{R}^{-}$, we have $\bar{b}_{0}\in
i\mathbb{R}^{+}$. Then (\ref{11-7-2}) gives
\[
\bar{a}_{0}\in \{z:\; \wp^{\prime}(z)\in i\mathbb{R}^{+}\} \subset
i\mathbb{R}^{-}\cup \{z:\;x_{2}>0\}.
\]
Since $\bar{a}_{0}\notin i\mathbb{R}^{-}$, we have $x_{2}(\bar{a}_{0})>0$ and
then $x_{2}(\bar{a}_{0})x_{2}(\bar{b}_{0})>0$. This is a contradiction to
$x_{2}(\bar{a}_{0})x_{2}(\bar{b}_{0})\leq0$.
\end{proof}

Now we are in a position to prove Theorem \ref{thm1}.

\begin{proof}
[Proof of Theorem 1.1]We just need to prove that \eqref{new213} has no
solutions for $\tau \in i\mathbb{R}^{+}$, i.e. $E_{\tau}$ is a rectangle.

Assume by contradiction that $(a,b)$ is a solution of \eqref{new213}. By
Lemmas~\ref{l1-22-1} and \ref{l2-22-1}, both $a$ and $b$ are in the interior
of $E_{\tau}$, and neither $a$ nor $b$ is on the coordinate axes. On the other
hand, it is well known (cf. \cite[Lemma 2.1]{CLW4}) that the Green function
$G$ in the rectangle $E_{\tau}$ satisfies%
\[
G_{x_{1}}(x_{1},x_{2})<0\text{ \ if }x_{1}\in(0,\tfrac{1}{2})\text{ and }%
x_{2}\in(-\tfrac{|\tau|}{2},\tfrac{|\tau|}{2});
\]%
\[
G_{x_{1}}(x_{1},x_{2})>0\text{ \ if }x_{1}\in(-\tfrac{1}{2},0)\text{ and
}x_{2}\in(-\tfrac{|\tau|}{2},\tfrac{|\tau|}{2});
\]%
\[
G_{x_{2}}(x_{1},x_{2})<0\text{ \ if }x_{2}\in(0,\tfrac{|\tau|}{2})\text{ and
}x_{1}\in(-\tfrac{1}{2},\tfrac{1}{2});
\]%
\[
G_{x_{2}}(x_{1},x_{2})>0\text{ \ if }x_{2}\in(-\tfrac{|\tau|}{2},0)\text{ and
}x_{1}\in(-\tfrac{1}{2},\tfrac{1}{2}).
\]
Together with Lemma \ref{lemma-7}, we conclude that $G_{z}(a)+G_{z}(b)\not =%
0$, which yields a contradiction with (\ref{new213}). This completes the proof.
\end{proof}


\begin{thebibliography}{99}                                                                                               %


\bibitem {A}L.V. Ahlfors; {\it Complex analysis, An introdution to the theory of
analytic functions one complex variable}. third edition.

\bibitem {Bartolucci}D. Bartolucci and G. Tarantello; \textit{Liouville type
equations with singular data and their applications to periodic multi-vortices
for the electro-weak theory}. Comm. Math. Phys. 229 (2002), 3-47.

\bibitem {new3}H. Br\'{e}zis and L. Nirenberg; \textit{Positive solutions of
nonlinear elliptic equations involving critical Sobolev exponents.} Comm. Pure
Appl. Math. \textbf{36} (1983), 437-477.

\bibitem {CLMP}E. Caglioti, P. L. Lions, C. Marchioro and M. Pulvirenti;
\textit{A special class of stationary flows for two-dimensional Euler
equations: a statistical mechanics description.} Comm. Math. Phys.
\textbf{143} (1992), 501-525.

\bibitem {CLW}{ C. L. Chai, C. S. Lin and C. L. Wang; \textit{Mean field
equations, Hyperelliptic curves, and Modular forms: I}. Cambridge Journal of
Mathematics \textbf{3} (2015), 127-274.}

\bibitem {CL-1}{ C. C. Chen and C. S. Lin; \textit{Sharp estimates for
solutions of multi-bubbles in compact Riemann surfaces}. Comm. Pure Appl.
Math. \textbf{55} (2002), 728-771.}




\bibitem {CL3}{ C. C. Chen and C. S. Lin; \textit{Mean field equation of
Liouville type with singular data: Topological degree}. Comm. Pure Appl. Math.
\textbf{68} (2015), 887-947.}

\bibitem {CLW4}C. C. Chen, C. S. Lin and G. Wang; \textit{Concentration
phenomena of two-vortex solutions in a Chern-Simons model}. Ann. Scuola Norm.
Sup. Pisa Cl. Sci. (5) \textbf{3} (2004), 367-397.





\bibitem {new1}Z. J. Chen, T. J. Kuo, C. S. Lin and C. L. Wang; \textit{Green
function, Painlev\'{e} VI equation, and Eisenstein series of weight one}. J.
Differ. Geom. to appear.

\bibitem {Choe}K. Choe; \textit{Asymptotic behavior of condensate solutions in
the Chern-Simons-Higgs theory.} J. Math. Phys. \textbf{48} (2007).

\bibitem {Eremenko1}A. Eremenko and A. Gabrielov; \textit{On metrics of curvature $+1$
with four conic singularites on tori and on the sphere}. arXiv: 1508.06510v2
[math. CV] 2015.


\bibitem {Lang}S.\ Lang; \textit{Elliptic Functions}. Graduate Text in
Mathematics \textbf{112}, Springer--Verlag 1987.

\bibitem {LW}{C. S. Lin and C. L. Wang; \textit{Elliptic functions, Green
functions and the mean field equations on tori}. Annals of Math. \textbf{172}
(2010), no.2, 911-954.}

\bibitem {CLW2}{C. S. Lin and C. L. Wang; \textit{Mean field equations,
Hyperelliptic curves, and Modular forms: II}. preprint 2015. arXiv:
1502.03295v2 [math.AP].}

\bibitem {LY}C. S. Lin and S. Yan; \textit{Existence of bubbling solutions for
Chern-Simons model on a torus.} Arch. Ration. Mech. Anal. \textbf{207} (2013), 353-392.

\bibitem {NT1}M. Nolasco and G. Tarantello; \textit{Double vortex condensates
in the Chern-Simons-Higgs theory.} Calc. Var. PDE. \textbf{9} (1999), 31-94.









\end{thebibliography}
\end{document}